\documentclass[11pt,amssymb,amsmath]{article}
\usepackage{epsf}
\usepackage{amssymb}
\usepackage{graphics}
\usepackage{amsmath}
\usepackage{amsthm}
\usepackage{amsfonts}
\usepackage{latexsym}
\usepackage{amssymb}
\usepackage{graphicx}
\begin{document}

\newcommand{\om}{\omega}

\newcommand{\beqn}{\begin{eqnarray}}
\newcommand{\eeqn}{\end{eqnarray}}
\newcommand{\fr}{\frac}
\newcommand{\nin}{\noindent}
\newcommand{\beq}{\begin{equation}}
\newcommand{\eeq}{\end{equation}}
\newtheorem{defn}[subsection]{Definition}
\newtheorem{thm}[subsection]{Theorem}
\newtheorem{prop}[subsection]{Proposition}
\newtheorem{cor}[subsection]{Corollary}
\newtheorem{remark}[subsection]{Remark}
\newtheorem{conjecture}[subsection]{Conjecture}

\newenvironment{rem}{\smallskip\noindent%
\refstepcounter{subsection}%
{\bf \thesubsection}~~{\sc Remark.}\hspace{-1mm}}
{\smallskip}

\newenvironment{ex}{\smallskip\noindent%
\refstepcounter{subsection}%
{\bf \thesubsection}~~{\sc Example.}}
{\smallskip}

\newenvironment{que}{\smallskip\noindent%
\refstepcounter{subsection}%
{\bf \thesubsection}~~{\sc Question.}\hspace{-1mm}}
{\smallskip}

\renewcommand{\theequation}{\arabic{section}.\arabic{equation}}

\newcommand{\Vir}{\text{\upshape Vir}}
\newcommand{\Rot}{\text{\upshape Rot}}
\newcommand{\Diff}{\text{\upshape Diff}}
\newcommand{\Vect}{\text{\upshape Vect}}
\newcommand{\Met}{\text{\upshape Met}}
\newcommand{\Metmu}{\Met_{\mu}}
\newcommand{\Vol}{\text{\upshape Dens}}
\newcommand{\SDiff}{Diff_\mu}
\newcommand{\SVect}{\Vect_\mu}
\newcommand{\Ham}{\text{\upshape Ham}}
\newcommand{\HamM}{\Ham_\omega(M)}
\newcommand{\ham}{\text{\upshape ham}}
\newcommand{\hamM}{\ham_\omega(M)}
\newcommand{\Symp}{\text{\upshape Symp}}
\newcommand{\symp}{\text{\upshape symp}}
\newcommand{\SDiffM}{\Diff_\mu(M)}
\newcommand{\SympM}{\Symp_\omega(M)}
\newcommand{\sympM}{\symp_\omega(M)}
\newcommand{\Diffmu}{\Diff_{\mu}}
\newcommand{\Isog}{\text{Iso}_g}
\newcommand{\DiffM}{\Diff(M)}
\newcommand{\MetM}{\Met(M)}
\newcommand{\MetmuM}{\Metmu(M)}
\newcommand{\VolM}{\Vol(M)}
\newcommand{\DiffmuM}{\Diffmu(M)}
\newcommand{\R}{\mathbb R}
\newcommand{\C}{\mathbb C}



\title{\Large\bf  Dynamics of symplectic fluids and point vortices}

\date{June 2011}

\author{\large\bf Boris Khesin\thanks{Department of Mathematics,
University of Toronto, Toronto, ON M5S 2E4, Canada;
e-mail: {\tt khesin@math.toronto.edu}}~}

\maketitle
\smallskip

{\hskip 1.5in{\it In memory of Vladimir Igorevich Arnold} }

\bigskip

\begin{abstract}
We present the Hamiltonian formalism for the Euler equation of 
symplectic fluids, introduce symplectic vorticity, 
and study related invariants. 
In particular, this allows one to extend D.~Ebin's long-time 
existence result for  geodesics on the symplectomorphism group
to metrics not necessarily compatible
with the symplectic structure. We also study the dynamics 
of symplectic point vortices, describe their symmetry groups and integrability.
\end{abstract} 



\bigskip
\bigskip



In 1966 V.~Arnold showed how the Euler equation describing dynamics of 
an ideal incompressible fluid  
on a Riemannian manifold can be viewed as a geodesic equation 
on the group of volume-preserving diffeomorphisms of this manifold \cite{arn}. Consider a similar problem for a symplectic fluid.

Let $(M^{2m},\om)$ be a closed symplectic manifold equipped with 
a Riemannian metric. A symplectic fluid filling $M$ is an ideal fluid 
whose motions 
preserve not only the volume element, but also the symplectic structure $\om$. 
(In 2D symplectic and ideal inviscid incompressible fluids coincide.) 
Such motions are
governed by the corresponding Euler-Arnold equation, i.e. the equation
describing geodesics on the infinite-dimensional group $\SympM$ 
of symplectomorphisms of $M$ with respect to the right-invariant $L^2$-metric. 
The corresponding problem of studying this dynamics
was posed in \cite{AK} (see Section IV.8).

Recently D.~Ebin \cite{Ebin} considered the corresponding  
Euler equation of the symplectic fluid 
and proved the existence of solutions for all times for compatible metrics 
and symplectic structures.
His proof uses the existence of a pointwise invariant transported by the flow, 
similar to the vorticity 
function in 2D. This symplectic vorticity allows one to proceed 
with the  existence proof in the symplectic case similarly to the 2D setting.

The purpose of this note is three-fold. First, 
we describe the Hamiltonian formalism of the  Euler-Arnold equation 
for symplectic fluids, the corresponding dual spaces, inertia operators, 
and Casimir invariants. This formalism manifests a curious duality to the incompressible case: its natural setting is a quotient space of 
$(n-1)$-forms for symplectic fluids vs. that of 
1-forms for incompressible ones. Second, we show the geometric 
origin of the symplectic vorticity arising
from the general approach to ideal fluids. 
In particular we prove that this quantity is a pointwise 
invariant for any metric, not only for a metric compatible 
with the symplectic structure, which allows one to extend the corresponding
long-time existence theorem for solutions of the symplectic Euler equation, see \cite{Ebin}.
We also present a variational description of the Hamiltonian stationary flows.
Finally, we present and study the problem of symplectic point vortices.
The Hamiltonian of the latter problem has more similarities
with the $N$-body problem in celestial mechanics than that 
of the classical problem of $N$ point vortices on the plane. We prove that 
this problem of symplectic vortices is completely integrable 
for $N=2$ point vortices in any dimension. 
We conjecture that it is not completely integrable for $N\ge 3$ in the 
Arnold-Liouville sense, similarly to the $N$-body problem in celestial mechanics.

To make the paper relatively self-contained we recall the setting
of an ideal fluid and classical point vortices and compare them to 
the new symplectic framework. While the analytical side of the problem
was explored in \cite{Ebin}, in this paper we are concerned with the geometric
and Hamiltonian aspects of the problem.

\bigskip

 
\section{Arnold's framework for the Euler  equation of an ideal incompressible fluid} 
 
We start with a brief reminder of the main setting of an ideal hydrodynamics.
Consider the Euler equation for an inviscid incompressible fluid 
filling some domain M in $\mathbb R^n$. The fluid motion is described 
as an evolution of the fluid velocity field $v(t,x)$  in $M$ which is governed by  the 
{\it classical Euler equation}:
\begin{equation}\label{ideal}
\partial_t v+v\cdot \nabla v=-\nabla p\,.
\end{equation}
Here the field $v$ is assumed to be divergence-free (${\rm div} v=0$)  
and tangent to the boundary of $M$.
The pressure function $p$ is defined uniquely modulo an additive constant 
by these restrictions on the velocity $v$.

The same equation describes the motion of an ideal incompressible fluid 
filling an arbitrary Riemannian manifold $M$ equipped with 
a volume form $\mu$ \cite{arn, EM}. 
In the latter case $v$ is a divergence-free vector field on $M$, while 
$v\cdot \nabla v$ stands for the Riemannian covariant derivative 
$\nabla_v v$ of the field $v$ in the direction of itself, while the divergence
$\mathrm{ div }~ v$ is taken with respect to the volume form $\mu$. 
 
\subsection*{The Euler equation as a geodesic flow}
Equation (\ref{ideal}) has a natural interpretation as a geodesic equation.
Indeed, the flow $(t,x)\to g(t,x)$ describing the motion of 
fluid particles is defined by its velocity field $v(t,x)$:
$$
\partial_t g(t,x)=v(t,g(t,x)), \,\, g(0,x)=x.
$$
The chain rule immediately gives
$\partial^2_{tt} g(t,x)=(\partial_t v+v\cdot \nabla v)(t,g(t,x))$, and hence the 
Euler equation is equivalent to 
$$
\partial^2_{tt} g(t,x)=-(\nabla p)(t,g(t,x)),
$$
while the incompressibility condition is $\det(\partial_x g(t,x))=1$.
The latter form of the Euler equation (for a smooth flow $g(t,x)$)
means that it describes a geodesic on the set $\SDiffM$ of 
volume-preserving diffeomorphisms of the manifold $M$.
Indeed, the acceleration of the flow, $\partial^2_{tt} g$, 
being given by a gradient, $-\nabla p$,
is $L^2$-orthogonal to all divergence-free fields, which constitute 
the tangent space to this set $\SDiffM$. 

In the case of any Riemannian  manifold $M$ the Euler equation defines 
the geodesic flow on the group of volume-preserving diffeomorphisms of $M$ 
with respect to the right-invariant $L^2$-metric on $\SDiffM$. 
(A proper analysis framework deals with the Sobolev $H^s$ spaces of
vector fields with $s>\frac n2+1$ and it is described in \cite{EM}.)

In the Euler equation for a symplectic fluid the same derivation replaces 
the gradient term $\nabla p$ in \eqref{ideal}
by a vector field $q$ which is $L^2$-orthogonal to the space 
of all symplectic vector fields and can be found in terms of 
the Hodge decomposition.
\bigskip



\subsection*{Arnold's framework for the Euler-type
equations} 
In \cite{arn} V. Arnold suggested the following general framework for the Euler
equation on an arbitrary group describing a geodesic flow with 
respect to a suitable one-sided invariant Riemannian metric on this group.

Consider a (possibly infinite-dimensional) Lie group $G$,
which can be thought of as the configuration
space of some physical system. The tangent space at the identity of the
Lie group $G$ is the corresponding Lie algebra
$\mathfrak g$. Fix some (positive definite) quadratic form, the ``energy,''  
$E(v)=\frac 12\langle v,Av\rangle$
on  $\mathfrak g$ and  consider right translations
of this quadratic form to the tangent space at any point of the group
(the ``translational symmetry'' of the energy).
This way the energy defines a  right-invariant
Riemannian metric on the group $G$.
The geodesic flow on $G$ with respect to this energy metric
represents the extremals of the least action principle, i.e.,
the actual motions of our physical system.

To describe a geodesic on the  Lie group with an initial velocity 
$v(0)=\xi$, we transport its velocity vector at any moment  $t$
to the identity of the group (by using the right translation).
This way we obtain the evolution law for $v(t)$, given by
a (non-linear) dynamical system $dv/dt=F(v)$
on the Lie algebra $\mathfrak g$. The system on the Lie algebra $\mathfrak g$,
describing the evolution
of the velocity vector  along a geodesic in a right-invariant metric
on the Lie group $G$, is called the {\it Euler} (or {\it Euler-Arnold}) 
{\it equation} corresponding to
this metric on $G$.

%
%

The Euler equation on the Lie algebra $\mathfrak g$ has 
a more explicit Hamiltonian reformulation
on the dual space $\mathfrak g^*$ of the Lie algebra $\mathfrak g$.
Identify the Lie algebra and its dual with the help of the above quadratic form
$E(v)=\frac 12\langle v,Av\rangle$.
This identification   $A:\mathfrak g\to \mathfrak g^*$ (called the {\it inertia
operator}) allows one to rewrite the Euler equation on the dual space
 $\mathfrak g^*$.

It turns out that the Euler equation on $\mathfrak g^*$
is Hamiltonian with respect to the natural Lie--Poisson structure on
the dual space \cite{arn}.
Moreover, the corresponding Hamiltonian function is  the energy
quadratic form lifted from the Lie algebra to its
dual space by the same identification:
$H(m)=\frac 12\langle A^{-1}m,m\rangle$, where $m=Av$.
Here we are going to take it as the {\it definition}
of the Euler equation on the dual space $\mathfrak g^*$.

\begin{defn} {\rm (see, e.g., \cite{AK})}
The  {\rm Euler equation on} $\mathfrak g^*$, corresponding to the 
right-invariant metric $E(m)=\frac 12\langle Av,v\rangle$ on the group,
is given by the following explicit formula:
\begin{equation}\label{linEuler}
\frac{dm}{dt}=-{\rm ad}^*_{A^{-1}m}m,
\end{equation}
as an evolution of a point $m\in \mathfrak g^*$.
\end{defn}

\bigskip


\section{Hamiltonian approach to incompressible fluids}
In this section we recall the Hamiltonian framework for the classical 
Euler hydrodynamics of an incompressible fluid, which we are going 
to generalize to symplectic fluids in the next section.

Let $M$ be an $n$-dimensional Riemannian manifold with a volume form $\mu$ 
and filled with an ideal incompressible fluid.
The corresponding Lie group $G=\SDiffM$ is the group of volume-preserving 
diffeomorphisms of $M$. 
The corresponding Lie algebra $\mathfrak g={\SVect}(M)$ 
consists of  divergence-free vector field in $M$:  
${\SVect}(M)=\{v\in \Vect(M)~|~L_v\mu=0\}$. 

Equip the group $G$ with the right-invariant metric by using
the $L^2$-product on divergence-free vector fields on $M$.
The formalism of the hydrodynamical Euler equation can be summarized 
in the following theorem:

\begin{thm} {\rm (see \cite{AK})}
a) The dual space
$\mathfrak g^*=\Omega^1(M)/\Omega^0(M)$ 
is the space of cosets of 1-forms on $M$ 
modulo exact 1-forms. The group coadjoint action is the change of coordinates
in the 1-form, while the corresponding Lie algebra coadjoint action is the 
Lie derivative along a vector field: ${\rm ad}^*_v=L_v$. Its action on cosets
 $[u]\in \Omega^1/d\Omega^0$ is well-defined.

b) The inertia operator is lifting the indices: $A: v\mapsto [v^\flat]$, 
where one considers the coset of the 1-form $v^\flat$ on $M$.
More precisely, for a manifold $M$ equipped with a Riemannian metric
$(.,.)$ one defines the 1-form $v^\flat$ as the pointwise inner product with 
vectors of the velocity field $v$:
$v^\flat(\eta): = (v,\eta)$ for all $\eta\in T_xM$.

c) The Euler equation (\ref{linEuler}) on the dual space has the form
$$
\partial_t[u]=-L_v[u],
$$
where $[u]\in \Omega^1(M)/\Omega^0(M)$ stands for a coset of 1-forms and
the vector field $v$ is related with a 1-form $u$ by means of 
a Riemannian metric on $M$: $u=v^\flat$. 
\end{thm}

The idea of the proof is that the 
map $v\mapsto i_v\mu$ provides an isomorphism of
the space of divergence-free vector fields and 
the space of closed $(n-1)$-forms on $M$: 
${\SVect}(M)\cong Z^{n-1}(M)$, since $d(i_v\mu)=L_v\mu=0$.
Then the dual space to the Lie algebra $\mathfrak g=Z^{n-1}$ 
is $\mathfrak g^*= \Omega^1/d\Omega^0$, and the pairing is
$$
\langle v,[u]\rangle:=\int_M (i_vu)\,\mu\,.
$$
The coadjoint action is the change of coordinates in differential forms.
The substitution of the inertia and coadjoint operators to the formula (\ref{linEuler})
yields the Euler equation.

\begin{remark} \upshape 
The Euler equation for a coset $[u]$
can be rewritten as an equation for a representative 1-form modulo
a function differential $dp$:
$$
\partial_t u +L_v u=-dp,
$$
where one can recognize the elements of the Euler equation (\ref{ideal})
for an ideal fluid.
\end{remark}

Note that each coset $[u]$ has a unique 1-form $\bar u\in [u]$
related to a {\it divergence-free} vector field by means of the metric. 
This is a coclosed 1-form:  $\delta \bar u=0$ on $M$.
Such a choice of a representative $\bar u\in [u]$
defines the pressure $p$ uniquely (modulo a constant), since
$\Delta p:=\delta d \,p$ gets prescribed for each time $t$.

\begin{remark} \upshape 
Define the vorticity 2-form $\xi=du$. It is well defined for a coset $[u]$.
The vorticity form of the Euler equation is
$$
\partial_t \xi=-L_v \xi,
$$
which means that the vorticity 2-form $\xi$ is transported by the flow.
The frozenness of the vorticity form allows one to define 
various invariants of the  hydrodynamical Euler equation.
\end{remark}

\begin{remark} \upshape 
Recall, that the Euler equation of an ideal fluid (\ref{ideal}) filling a  
three-dimensional simply connected manifold $M$
has the helicity (or Hopf) invariant. 
Topologically helicity in 3D describes the mutual
linking of the trajectories of the vorticity field ${\rm curl }\,v$, and
has the form $J(v) =\int_{M^3} ({\rm curl}\,v,~v)~\mu$.

For an ideal 2D fluid one
has an infinite number of conserved quantities, so called {\it enstrophies}:
$$
J_k(v) = \int_{M^2} ({\rm curl}~ v)^k~\mu
\qquad{\rm for } ~k=1,2,\dots,
$$
where curl~$v:=du/\mu={\partial v_1}/{\partial x_2}
-{\partial v_2}/{\partial x_1}$ is a {\it vorticity function}
of a 2D flow.

It turns out that enstrophy-type integrals   exist 
for all even-dimensional flows,  and so do
 helicity-type integrals for all
odd-dimensional ideal fluid flows, see e.g. \cite{Ser, AK}.
The invariance of the helicity and enstrophies follows, in fact, from their 
coordinate-free definition: they are 
invariant with respect to volume-preserving coordinate changes, and hence, are first integrals
 of the corresponding Euler equations.
\end{remark}

\bigskip


\section{Hamiltonian approach to symplectic fluids}

Let $(M^{2m},\omega)$ be a closed symplectic manifold of dimension $n=2m$ which is also
equipped with a Riemannian metric. Consider the dynamics of 
a fluid in $M$ preserving the symplectic 2-form $\omega$.

The configuration space of a symplectic fluid on $M$ is the 
Lie group  $G=\SympM$ and it is equipped with the right-invariant 
$L^2$-metric. It is a subgroup of the group of volume-preserving 
diffeomorphisms: 
$G=\SympM\subset \SDiffM$, where $\mu=\omega^m/m$ is the symplectic volume. 
We will also assume that the metric volume coincides with the symplectic one, 
but do not assume that 
the symplectic structure and metric are compatible.

The corresponding Lie algebra is 
$\mathfrak g={\rm symp}_\omega(M)=\{v\in\Vect(M)~|~L_v\omega=0\}$.

\begin{thm} 
a) The dual space to $\mathfrak g={\sympM}$
is $\mathfrak g^*={\rm symp}^*_\omega(M)\cong\Omega^{n-1}(M)/d\Omega^{n-2}(M)$.
The pairing between  $v\in \mathfrak g=\sympM$ and $\alpha\in \Omega^{n-1}(M)$ 
is given by the formula
\begin{equation}\label{sympair}
\langle  v, \alpha\rangle:=\int_M \alpha\wedge\,i_v\omega \,,
\end{equation}
and it is well-defined on cosets $[\alpha]\in \Omega^{n-1}(M)/d\Omega^{n-2}(M)$.

The algebra coadjoint action is the Lie derivative: ${\rm ad}^*_v=L_v$ and 
it is well-defined on cosets.

b) The inertia operator $A$ 
is lifting the indices and wedging with $\omega^{m-1}$, i.e. 
$A: v\mapsto [v^\flat\wedge \,\omega^{m-1}]$, where $v^\flat$ is the 1-form
on $M$ and one takes the corresponding coset of the form 
$\alpha= v^\flat\wedge\, \omega^{m-1}\in \Omega^{2m-1}(M)$.

c) The Euler equation (\ref{linEuler}) on the dual space has the form
$$
\partial_t[\alpha]=-L_v[\alpha],
$$
where the vector field $v$ is related to the $(n-1)$-form $\alpha$ 
by means of the inertia operator:
$\alpha=v^\flat\wedge \,\omega^{m-1}$ 
and $ [\alpha]\in \Omega^{n-1}/d\Omega^{n-2}$ 
stands for its coset. 
\end{thm}

\begin{proof} 
The Lie algebra $\sympM$
is naturally isomorphic to the space of closed 1-forms $Z^1(M)$. Indeed,
the requirement for a field $v$ to be symplectic, 
$0=L_v\omega=i_vd\omega+di_v\omega=di_v\omega$, is equivalent 
to closedness of the 1-form $i_v\omega$, while
each closed 1-form on $M$ can be obtained as  $i_v\omega$
due to the nondegeneracy of the symplectic form $\omega$.

The isomorphism $\sympM\cong Z^1(M)$ implies that $\symp^*_\omega(M)\cong \Omega^{n-1}(M)/d\Omega^{n-2}(M)$
for the (regular) dual spaces with pairing given by wedging of differential forms.

The explicit expression for the inertia operator follows from 
the following transformations:
$$
E(v) =\frac {1}{2} \int_M(v,v)\,\frac{\omega^m}{m}
=\frac {1}{2m} \int_M (i_vv^\flat)\,\omega^m
=\frac {1}{2m} \int_M v^\flat\wedge \,i_v\omega^m
=\frac 12 \int_M v^\flat\wedge \,\omega^{m-1}\wedge \, i_v\omega\,.
$$
Hence the energy $E(v)=\frac 12 \langle v, Av\rangle$ is given by the inertia
operator $Av=[v^\flat\wedge \,\omega^{m-1}]$, due to the pairing 
\eqref{sympair} between symplectic fields and cosets on $(n-1)$-forms.

Note that for $m=1$, the 2D case, the inertia operator $A: v\mapsto [v^\flat]$ 
coincides with that for an ideal incompressible 2D fluid.
\end{proof}

\begin{remark} \upshape 
In representatives $(n-1)$-forms one gets the equation modulo an exact form:
\begin{equation}\label{symEuler}
\partial_t \alpha +L_v \alpha=-d\beta,
\end{equation}
where $\beta \in \Omega^{n-2}$ and by applying the inverse inertia operator 
one obtains the symplectic Euler equation.
One can show that each coset $[\alpha]$ has 
a unique 1-form $\bar \alpha\in [\alpha]$
related to a symplectic vector field by means of the metric.
\end{remark}


\subsubsection*{Symplectic vorticity}

\begin{defn}  
Define the {\rm symplectic vorticity} for a symplectic vector field $v$
to be the $n$-form 
$\xi:=d\alpha\in \Omega^{n}(M)$, where $n=2m=\dim M$, while the field $v$ 
and the $(n-1)$-form $\alpha$  
are related by the inertia operator: $[\alpha]=[v^\flat\wedge \,\omega^{m-1}]$.

The  {\rm symplectic vorticity function} $\nu$ 
is the  ratio $\nu:=d\alpha/\omega^m$ of the symplectic vorticity form 
and the symplectic volume. 
\end{defn}

\begin{prop}
Both the symplectic vorticity form $\xi$ and symplectic vorticity 
function $\nu$ are  transported by the symplectic flow.
\end{prop}

Indeed, by taking the differential of  
both sides of the Euler equation \eqref{symEuler}
we obtain the vorticity form of the symplectic  Euler equation 
$$
\partial_t \xi =-L_v \xi\,,
$$
which expresses the fact that the $n$-form $\xi=d\alpha$ 
is transported by the symplectic flow. The symplectic vorticity function
$\nu$ is also transported 
by the flow, just like in the ideal 2D case:
$$
\partial_t \nu =-L_v \nu\,,
$$
since the symplectic volume $\omega^m$ is invariant under the flow. 

This geometric observation allows one to extend Ebin's theorem to 
symplectic manifolds with metrics not necessarily compatible with 
symplectic structures.

\begin{cor}{\rm (cf. \cite{Ebin})}
The solutions of the symplectic Euler equation \eqref{symEuler} on
a closed Riemannian symplectic manifold $(M^{2m},\omega)$
in spaces $H^s$ with $s>\frac n2+1$ exist for all $t$ for any metric 
whose volume element coincides with the symplectic volume.
\end{cor}

\begin{proof}
The corresponding theorem is proved by D.~Ebin in \cite{Ebin}
for metrics $g$ compatible with the symplectic structure $\omega$, i.e. 
for which there is an almost complex structure $J$, 
so that $\omega(v,w)=g(Jv,w)$, see e.g. \cite{HZ}.

The proof is based on the existence of an invariant quantity, similar
to the vorticity of an ideal 2D fluid, which allowed one 
to reduce the existence questions in the symplectic case 
to similar questions in the 2D case and to adapt the corresponding 
2D long-time existence proof to the symplectic setting. 
The compatibility is used in the proof of the invariance of this quantity.

However, the symplectic vorticity $\nu$ is exactly the invariant 
quantity defined in the paper by D.~Ebin, and the above geometric point of view
proves its frozenness into symplectic fluid without the requirement of 
compatibility. 
\end{proof}

\begin{remark}\upshape
The transported symplectic vorticity also allows one to obtain 
infinitely many conserved quantities (Casimirs):
$$
I_k(\alpha)=\int_M\nu^k\,\omega^m\qquad {\text for~any~}\quad k=1,2,3,...,
$$
which are invariants of the symplectic Euler equation for any metric and symplectic form on $M$.
\end{remark}

\begin{remark}\upshape
Since symplectic vector fields form a Lie subalgebra in divergence-free ones,
$\sympM\subset\SVect(M)$ for $\mu=\omega^m/m$, there is the natural projection
of the dual spaces $\SVect^*\to\symp^*_\omega$, i.e. the projection
$\Omega^1/d\Omega^0\to \Omega^{n-1}/d\Omega^{n-2}$  given by $[u]\mapsto [u\wedge\omega^{m-1}]$. This projection respects the coadjoint action, while the vorticity 2-form of an ideal fluid under the projection becomes the symplectic vorticity.
This is yet one more way to check that the symplectic vorticity is frozen into a symplectic flow for any metric on $M$. 
\end{remark}

\medskip

\begin{remark}\upshape
We also mention the necessary changes to describe Hamiltonian fluids.
Now for a closed symplectic manifold $(M^{2m},\omega)$  consider the group of
Hamiltonian diffeomorphisms $G=\HamM$, i.e. those diffeomorphisms that are
attainable from the identity by Hamiltonian vector fields.

Its Lie algebra is $\mathfrak g=\hamM=\{v\in \Vect(M)~|~i_v\omega=dH\}$.
By definition, this Lie algebra is naturally isomorphic to the space 
of exact 1-forms,
$\hamM\cong d\Omega^0(M)\cong C^\infty(M)/\{ {\text constants} \}$.

Then its dual space is $\mathfrak g^*=\hamM^*=\Omega^{n-1}(M)/Z^{n-1}(M)$, i.e. the space
of all $(n-1)$-forms modulo closed ones. The coadjoint action is again
the Lie derivative ${\rm ad}^*_v=L_v$.
The energy and inertia operators are the same as for symplectic fluids.

Note that the dual space 
$\mathfrak g^*=\Omega^{n-1}(M)/Z^{n-1}(M)$ is naturally 
isomorphic to exact $n$-forms $d\Omega^{n-1}(M)$ on $M$, 
since for  $\alpha\in \Omega^{n-1}(M)$
the map $[\alpha]\mapsto d\alpha$ is an isomorphism.
Hence the Euler equation $\partial_t[\alpha]=-L_v[\alpha]$
is now {\it equivalent} to its vorticity formulation:
$$
\partial_t \nu=-L_v \nu\,,
$$ 
for the symplectic vorticity function $\nu=d\alpha/\omega^m$.

Finally, note that for a Hamiltonian field $v$ one can rewrite the above  equation 
with the help of the Poisson bracket on the symplectic manifold $(M,\omega)$ as follows
$$
\partial_t \nu=\{\psi,\nu\}\,,
$$ 
where $\psi$ is the Hamiltonian function for the velocity field $v$, 
while $\nu$ is its symplectic vorticity function and they are related
via $\Delta \psi=\nu$.
Indeed, the latter relation between $\psi$ and $\nu$ is equivalent to the
relation furnished by the inertia operator:
$dv^\flat\wedge\omega^{m-1}=\nu\cdot\omega^{m}$.
This shows that the stream-function formulation of the 2D Euler equation
is valid for the Hamiltonian fluid in any dimension.
\end{remark}

\subsubsection*{Hamiltonian steady flows}

Steady solutions ($\partial_t \nu=0$) 
to the  Euler equation for Hamiltonian fluids are given by 
those vector fields on $M$ whose Hamiltonians $\psi$ Poisson commute 
with their Laplacians $\nu=\Delta \psi$:
$$
\{\psi,\nu\}=0\,.
$$
For generic Hamiltonians in 2D this means that they are functionally 
dependent with their Laplacians, cf. \cite{arn}, the fact used by V.Arnold 
in the 60s to obtain stability conditions in ideal hydrodynamics.
In higher dimensions this merely means that that the two functions $\psi$ and 
$\nu=\Delta \psi$ are in involution with respect to the natural Poisson bracket on $M$.
For generic Hamiltonians in 4D this implies that the corresponding steady flows
represent integrable systems with two degrees of freedom, similarly to 
the case of an incompressible 4D fluid  studied in \cite{GK}.

For a Riemannian symplectic manifold $M$
consider the Dirichlet  functional 
$$
D(\psi):=\int_M({\rm sgrad~} \psi,{\rm sgrad~} \psi)\,\omega^m
$$ 
on Hamiltonian functions on $M$ obtained from a given function $\psi_0$ by the action
of Hamiltonian diffeomorphisms.

\begin{prop}
Smooth extremals (in particular, smooth minimizers) of the Dirichlet functional with respect to the action 
of the Hamiltonian diffeomorphisms on functions 
are given by Hamiltonians of steady vector fields, 
i.e. Hamiltonian functions satisfying
$$
\{\psi,\Delta\psi\}=0\,.
$$
\end{prop}

\begin{proof}
Indeed, the Dirichlet  functional is (up to a factor)  the 
energy functional $E=\frac 1{2m}\int_M(v,v)\,\omega^m$ on Hamiltonian fields.
The variational problem described above is funding extrema of $E$ 
on the group adjoint orbit containing ${\rm sgrad~} \psi_0$  in the Lie algebra 
$\HamM$. These extrema are in 1-1 correspondence with extrema of 
the energy functional on the coadjoint orbits, see the general theorem 
in \cite{AK}, section II.2.C. In turn, extrema for the kinetic energy 
on coadjoint orbits 
in $\ham^*_\omega(M)$ are given by stationary Hamiltonian fields.
\end{proof}

Note that for metrics compatible with the symplectic structure
the above functional becomes the genuine Dirichlet  functional on functions:
$D(\psi):=\int_M(\nabla \psi, \nabla \psi)\,\omega^m$.
There also is a similar variational and direct descriptions for steady
symplectic fields, i.e. for the group $\SympM$.


\section{Classical and symplectic point vortices}
In this section we recall several facts about the classical problem 
of point vortices in the 2D plane and 
consider the symplectic analog of  point vortices
for higher-dimensional symplectic spaces.
For an ideal 2D fluid the systems of 2 and 3 point vortices are 
known to be completely integrable, while systems of $\ge 4$ point vortices are not.
It turns out that the corresponding evolution of symplectic vortices for $2m>2$ 
is integrable  for $N=2$ and presumably is non-integrable for $N\ge 3$.
This generalized system of symplectic vortices in a sense looks more
like a many-body problem in space, which is non-integrable already 
in three-body case.

Consider the 2D Euler equation  in the vorticity form:
$$
\dot \nu =\{\psi,\nu\}\,,
$$
where $\nu$ is the vorticity function and
the stream function (or Hamiltonian) $\psi$ of the flow satisfies
$\Delta \psi=\nu$. 
The same equation governs the evolution of the symplectic vorticity $\nu$
of a Hamiltonian fluid  on any symplectic manifold $(M^{2m},\omega)$,
where  the  symplectic vorticity $\nu$ and the Hamiltonian function
 $\psi$ of the flow are related in the same way: $\Delta \psi=\nu$ and
 $\Delta$ is the  Laplace-Beltrami operator, see the preceding section.

Now we consider the symplectic space $\R^{2m}$
with the standard symplectic structure 
$\omega=\sum_{\alpha=1}^m dx_\alpha\wedge dy_\alpha
=d\tilde x\wedge d\tilde y$. 
Let symplectic vorticity $\nu$ be supported in $N$ point vortices:
$\nu=\sum^N_{j=1} \Gamma_j\,\delta(\tilde z-\tilde z_j)$, where
$\tilde z_j=(\tilde x_j, \tilde y_j)$ 
are coordinates of the vortices on the space $\R^{2m}=\C^m$ with $m \ge 1$.

\begin{thm}
The evolution of vortices according to the Euler equation 
is described by the system
$$
\Gamma_j \dot x_{j,\alpha}
=\frac{\partial \mathcal H}{\partial y_{j,\alpha}},\qquad
\Gamma_j \dot y_{j,\alpha}
=-\frac{\partial \mathcal H}{\partial x_{j,\alpha}},\qquad 
1\le j\le N, \quad 1\le \alpha\le m\,.
$$
This is a Hamiltonian system on 
$(\R^{2m})^N$ with the Hamiltonian function
$$
\mathcal H=2\cdot C(2m)\sum^N_{j<k}\Gamma_j\Gamma_k \, |\tilde z_j-\tilde z_k|^{2-2m} 
\quad \text{ for } m>1\, \quad \text{ or }
$$
$$
\mathcal H=-\frac{1}{4\pi}\sum^N_{j<k}\Gamma_j\Gamma_k \, 
\ln |\tilde z_j-\tilde z_k|^{2} 
\quad \text{ for } m=1\,.
$$
\end{thm}

Here the distance $|\tilde z_j-\tilde z_k|$ is defined in $(\R^{2m})^N$, 
the constant $C(2m)$ is the constant of the Laplace fundamental solution
in $\R^{2m}$, and the Poisson structure given by the bracket
$$
\{f,g\}=\sum^N_{j=1}\frac{1}{\Gamma_j}
\left( \frac{\partial f}{\partial \tilde x_j}
\frac{\partial g}{\partial \tilde y_j}
-\frac{\partial f}{\partial \tilde y_j}
\frac{\partial g}{\partial \tilde x_j}\right)
=\sum^N_{j=1}\frac{1}{\Gamma_j}\sum^m_{\alpha=1}
\left( \frac{\partial f}{\partial x_{j,\alpha}}
\frac{\partial g}{\partial y_{j,\alpha}}
-\frac{\partial f}{\partial y_{j,\alpha}}
\frac{\partial g}{\partial x_{j,\alpha}}\right)\,.
$$

The case $m=1$ goes back to Kirchhoff. 
The case $m>1$ apparently did not appear in the literature before.

\begin{proof}  Any (non-autonomous) Hamiltonian equation
$\dot z=\text{ sgrad }H(t,z)$ in a symplectic manifold $M$ 
has an alternative (Liouville)
version: $\dot \rho=\{H,\rho\}$ for any smooth function $\rho$ on $M$.
The reduction of the Liouville version to the Hamiltonian one is obtained by taking the limit
as $\rho$ tends to a delta-function supported at a given point $z\in M$.

In particular,  the vorticity equation
$\dot \nu =\{\psi,\nu\}$ describes a Hamiltonian equation on
$\R^{2m}$ with instantaneous Hamiltonian function $\psi$. By assuming 
that $\nu$ is of the form  
$\nu=\sum^N_{j=1} \Gamma_j\,\delta(\tilde z-\tilde z_j),\quad \tilde z\in \R^{2m}$, 
one obtains the  instantaneous Hamiltonian
$$
\psi=\Delta^{-1}\nu=C(2m)\sum^N_{j=1}\Gamma_j \,|\tilde z-\tilde z_j|^{2-2m}\,.
$$
Thus the corresponding Hamiltonian form of the vorticity equation 
for point vortices
is $\dot {\tilde z}_j= \text{ sgrad } \psi|_{\tilde z=\tilde z_j}$.
Now the straightforward differentiation of $\psi$ at $\tilde z=\tilde z_j$ 
(in which one discards the singular term at $\tilde z_j$ itself) 
shows that the equation for 
$\dot {\tilde z}_j$ coincides with the Hamiltonian vector field 
for the Hamiltonian $\mathcal H$ on $(\R^{2m})^N$.

For instance, for the standard complex structure $\mathbf J$ 
in $\R^{2m}=\C^m$, $m>1$ one has 
$$
 \text{ sgrad } \psi|_{\tilde z=\tilde z_j}
=\mathbf J \text{ grad }|_{\tilde z=\tilde z_j}\left(C(2m)\cdot\sum^N_{k=1, k\not=j}\Gamma_k \,|\tilde z-\tilde z_k|^{2-2m}\right)
=\frac{1}{\Gamma_j}\,\mathbf J\,\frac{\partial \mathcal H}{\partial \tilde z_j}\,.
$$
as required.
\end{proof}

\begin{thm}
The above dynamical system of symplectic vortices 
is invariant with respect to the group $E(2m):=U(m)\ltimes \R^{2m}$ 
of unitary motions of $\R^{2m}=\C^m$. 
The corresponding $m^2+2m$ conserved quantities, 
which commute with $\mathcal H$,
 are:
$$
Q_\alpha=\sum^N_{j=1} \Gamma_j x_{j,\alpha}\,, \qquad
P_\alpha=\sum^N_{j=1} \Gamma_j  y_{j,\alpha}\,\qquad\text{~for~} 1\le \alpha\le m\,,
$$
$$
F^+_{\alpha\beta}
=\sum^N_{j=1} \Gamma_j(x_{j,\alpha}x_{j,\beta}
  +y_{j,\alpha}y_{j,\beta})\,\qquad\text{~for~} 
 1\le \alpha\le\beta\le m\,,\quad\text{~and~}
$$
$$ 
F^-_{\alpha\beta}=\sum^N_{j=1} \Gamma_j(x_{j,\alpha}y_{j,\beta}
                        -x_{j,\beta}y_{j,\alpha})\qquad\qquad\text{~for~} 
 1\le \alpha < \beta\le m\,.
$$
\end{thm}

\begin{proof}  The standard complex structure $\mathbf J$ in $\R^{2m}=\C^m$ 
is compatible with the symplectic and Euclidean structures on the space. 
Since the motions $E(2m):=U(m)\ltimes \R^{2m}$ 
preserve the Euclidean and complex structures,
they also preserve the symplectic one and hence preserve the equation, 
which is defined in terms of these structures.
The above quantities $Q_\alpha$ and $P_\alpha$ are generators of translations in the plane $(x_\alpha,y_\alpha)$, 
while $F^\pm_{\alpha\beta}$ generate unitary rotations in the space 
$( x_\alpha,y_\alpha, x_\beta,y_\beta)$.

Indeed, rewrite the Poisson structure and integrals in the $(z,\bar z)$-variables for $z=x+iy$. Namely, for  the Poisson structure
$2\frac{\partial}{\partial x}\wedge \frac{\partial}{\partial y}
=i\frac{\partial}{\partial \bar z}\wedge \frac{\partial}{\partial z}$
consider, respectively, the quadratic and linear functions 
$$ 
F_{\alpha\beta}:=-iz_\alpha\bar z_\beta \qquad \text{~and~}\qquad 
R_\alpha:=z_\alpha
$$ 
for $ 1\le \alpha,\beta\le m$. 
Note that 
$$\overline F_{\beta\alpha}=-F_{\alpha\beta} \qquad\text{~and~}\qquad
\{F_{\alpha\beta},F_{\beta\gamma}\}=F_{\alpha\gamma}\,,
$$
that is the Hamiltonians $F_{\alpha\beta}$ generate the unitary Lie
algebra $u(m)$ with respect to the Poisson bracket.
The corresponding Hamiltonian fields for $R_\alpha$ 
generate translations in the $\C$-lines $(z_\alpha)$,
while $F_{\alpha\beta}$ generate 
unitary rotations in the $\C$-planes  $(z_\alpha, z_\beta)$ preserving 
the norm $\langle z,z\rangle:=\sum_\alpha z_\alpha\bar z_\alpha$.
Assuming the summation in $j=1,...,N$ with weights $\Gamma_j$ we see
that $Q_\alpha$ and $P_\alpha$ are real and imaginary parts for $R_\alpha$,
while $F^\pm_{\alpha\beta}$ so are for the complex functionals 
$F_{\alpha\beta}$.

Since the corresponding Hamiltonian flows yield unitary motions of the 
space $\C^m$, the corresponding transformations commute 
with the equation of symplectic point vortices. 
\end{proof}

\begin{remark}  \upshape 
As we mentioned, the quadratic functionals $F_{\alpha\beta}$ 
form the Lie algebra $u(m)$ with respect to the 
Poisson bracket, while $Q_\alpha$ and $P_\alpha$ form $\R^{2m}$.
Furthermore, since $\{Q_\alpha,P_\alpha\}=\sum^N_{j=1} \Gamma_j$,
together the functionals $F_{\alpha\beta},Q_\alpha,$ and $P_\alpha$ form the central extension 
of the Lie algebra for the semi-direct product 
group $E(2m)=U(m)\ltimes \R^{2m}$ of unitary motions.
\end{remark}

One can use the above functionals  to construct  involutive integrals.

\begin{cor}\label{invint}
Functionals ${\mathcal H}, \,F^+_{\alpha\alpha}$ and $Q_\alpha^2+P_\alpha^2$ for $1\le \alpha\le m$  
provide $2m+1$ integrals in involution on $(\R^{2m})^N$.
\end{cor}

This follows from the commutation relations
$\{P_\alpha,F^+_{\alpha\alpha}\}=-2Q_\alpha$ and 
$\{Q_\alpha,F^+_{\alpha\alpha}\}=2P_\alpha$.

\begin{cor}
(a) In 2D (i.e. for $m=1$) the problems of $N=2$ and $N=3$ point vortices 
are completely integrable.

(b) The $N=2$ symplectic vortex problem in $\R^{2m}$ 
is integrable for any dimension $2m$.
\end{cor}

Indeed, for $m=1$ the phase space of $N=3$ vortices is 6 dimensional, 
while one has $2m+1=3$ independent involutive integrals.

For any dimension $m$ 
and $N$ vortices the phase space is $(\R^{2m})^N$. For integrability one 
needs $m\cdot N$ integrals in involution. Thus $2m+1$ 
integrals ${\mathcal H}, \,F^+_{\alpha\alpha}$ and $Q_\alpha^2+P_\alpha^2$
are sufficient for integrability of 
$N=2$ symplectic vortices for any $m$.

\begin{remark}  \upshape 
Note that the evolution of $N\ge 4$
point vortices in 2D is non-integrable \cite{Ziglin}.
When $m>1$ already for $N=3$ one does not have enough integrals 
for Arnold-Liouville integrability: for $m=2$ 
one has 5 integrals, but integrability requires 6.
\end{remark}

%

\begin{conjecture} 
The system of $N=3$ symplectic point vortices is  not completely 
integrable on $(\R^{2m})^N$ for $m>1$ in the Arnold-Liouville sense.
\end{conjecture} 

Note that for any two point vortices in  $\R^{2m}$, their velocities, given by 
$\text{ sgrad } \psi(\tilde z)$ at $\tilde z=\tilde z_j$ for $j=1,2$ lie in a fixed two-dimensional plane depending on the initial positions $\tilde z_1, \tilde z_2$. Thus for
$N=2$ the dynamics reduces to the 2D case. The motivation for the above conjecture is that for 
$N=3$ the vectors $\text{ sgrad } \psi(\tilde z_j)$ at $\tilde z_j$ do not necessarily lie in one and the same plane passing through the vortices $\tilde z_1, \tilde z_2, \tilde z_3 $ once $m>1$, i.e.
the problem becomes indeed higher-dimensional. 

In a sense, the systems of symplectic point vortices is somewhat similar to
many-body problem in higher dimensions, cf. \cite{rmont}. 
It would be interesting to describe the cases when 
the symplectic vortex problem is 
weakly integrable on $\R^{2m}$ in the sense of \cite{Butler} thanks to a large number of conserved quantities that are not in involution. 

\medskip

\begin{remark}  \upshape 
The evolution of point vortices on the sphere $S^2$ or 
the hyperbolic plane ${\mathbb H}^2$ is invariant for the groups
$SO(3)$ and $SO(2,1)$ respectively. The corresponding problems of 
$N\le 3$ vortices are integrable \cite{Kimura}. Furthermore,
one can make the first integrals in these cases deform to each other by 
tracing the change of curvature for the corresponding
symplectic manifolds (see the case of $SU(2)$, $E(2)$ and $SU(1,1)$ for $m=1$ in
\cite{Tadashi}). 

Similarly, one can consider the evolution of symplectic 
point vortices on the projective space $\mathbb{CP}^m$ or 
other homogeneous symplectic spaces with invariance with respect to the groups $SU(m)$ or $SU(k,l)$. 

Note that for a compact symplectic manifold $M^{2m}$ 
one needs to normalize the vorticity supported on $N$ point vortices by subtracting an appropriate constant:
$$
\nu=\sum^N_{j=1} \Gamma_j\,\delta(\tilde z-\tilde z_j)-C\,, \qquad
{\rm where}\qquad C=\frac{1}{{\rm Vol}(M)}\sum^N_{j=1} \Gamma_j\,.
$$
Indeed, for the existence of $\psi$ satisfying the equation $\Delta \psi=\nu$
on a compact $M$, the function $\nu$ has to have zero mean. 
\end{remark}

\begin{remark}  \upshape 
Recall that in an ideal hydrodynamics the  vorticity is geometrically
a 2-form, and 
for higher dimensional spaces $\R^n$ singular vorticity can be 
supported on submanifolds of codimension 2 (for instance, it describes
evolution of curves in $\R^3$). The corresponding Euler dynamics of the vorticity 2-form is nonlocal, since it requires finding ${\rm curl}^{-1}$.
The  localized induction approximation (LIA) of vorticity motion describes the
 filament (or binormal) equation, which is known to be integrable in $\R^3$.

However, for symplectic fluids there is no natural filament dynamics
since the symplectic vorticity is a $2m$-form, i.e. a form whose degree
is equal to the dimension of the manifold. Hence its singular version
is naturally supported at (symplectic) point vortices, rather than
on submanifolds of larger dimension.  
\end{remark}


\subsection*{Acknowledgments} 
I am indebted to G.~Misio\l ek 
and F.~Soloviev for fruitful suggestions. I am also grateful to the Ecole Polytechnique in Paris for hospitality during completion of this paper.
The present work was partially sponsored by an NSERC research grant.



\end{document}